\documentclass[12pt,a4paper]{article}

\usepackage{amssymb}
\usepackage{graphicx}
\usepackage{amsmath}
\usepackage{amsmath,amssymb}
\usepackage{amsfonts}
\usepackage{amsthm,amscd}
\usepackage[english]{babel}
\usepackage{graphicx}
\usepackage{amsmath}
\usepackage{amssymb}
\usepackage{amstext}

\newtheorem{theorem}{Theorem}
\newtheorem{corollary}[theorem]{Corollary}
\newtheorem{definition}[theorem]{Definition}

\newtheorem{lemma}[theorem]{Lemma}

\title{Large deviations for equilibrium measures and selection of subaction}

\author{Jairo K. Mengue \\
		\footnotesize{\texttt{jairo.mengue@ufrgs.br}}\\
	\footnotesize{Universidade Federal do Rio Grande do Sul }}

\begin{document}

\maketitle

\begin{abstract}
Given a Lipschitz function $f:\{1,...,d\}^\mathbb{N} \to \mathbb{R}$, for each $\beta>0$ we denote by $\mu_\beta$ the equilibrium measure of $\beta f$ and by $h_\beta$ the main eigenfunction of the Ruelle Operator $L_{\beta f}$. \newline
Assuming that $\{\mu_{\beta}\}_{\beta>0}$ satisfy a large deviation principle, we prove the existence of the uniform limit $V= \lim_{\beta\to+\infty}\frac{1}{\beta}\log(h_{\beta})$. Furthermore, the expression of the deviation function is determined by its values at the points of the union of the supports of maximizing measures. 
\newline
We study a class of potentials having two ergodic maximizing measures and prove that a L.D.P. is satisfied. The deviation function is explicitly exhibited and  does not coincide with the one that appears in the paper by Baraviera-Lopes-Thieullen which considers the case of potentials having a unique maximizing measure.
\end{abstract}

\section{Introduction}

We denote by $X$ the Bernoulli space $\{1,...,d\}^{\mathbb{N}},\,\, \mathbb{N}=\{1,2,3,...\},$ and by $\sigma$ the shift map acting on $X$. The metric considered satisfies $d_{\theta}(x,y) =\theta^{\min\{i,\,x_{i}\neq y_{i}\}}$, where $x=(x_1x_2x_3...)$, $y=(y_1y_2y_3...)$ and $\theta \in (0,1)$ is fixed. If $x_1,...,x_n \in \{1,...,d\}$ and $y=(y_1y_2y_3...)\in X,$ the notation $(x_1...x_ny)$ represents the element $(x_1x_2...x_ny_1y_2y_3....)\in X$. A cylinder is a subset of $X$ of the form $[x_1...x_n]:=\{(x_1...x_ny)\,|\, y\in X\}$.  We denote by
$C(X)$ the set of continuous functions from $X$ to $\mathbb{R}$ and by $P(X)$ the set of probabilities on $X$. 

Let $f:X\to \mathbb{R}$ be a Lipschitz function and,  for each $\beta >0$, denote by $L_{\beta f}$ the Ruelle operator associated with $\beta f$, which is defined by 
\[L_{\beta f}:C(X)\to C(X),\,\,\, (L_{\beta f}(w))(x)=\sum_{a\in\{1,...,d\}}e^{\beta f(ax)}w(ax).\]
 We denote by $\nu_{\beta}$ the eigenmeasure of $L_{\beta f}$, that is, the probability satisfying $\int L_{\beta f}(u)\, d\nu_\beta = e^{P(\beta f)}\int u\, d\nu_\beta$ for any continuous function $u:X \to \mathbb{R}$, and by $h_{\beta}$ the main eigenfunction of $L_{\beta f}$. More precisely, $h_\beta$ is Lipschitz,  $L_{\beta f}(h_\beta)=e^{P(\beta f)} h_\beta$ and $\int h_\beta \, d\nu_\beta= 1$. Let 
$g_\beta := \beta f + \log(h_\beta) - \log(h_\beta \circ \sigma) - P(\beta f)$. The functions $g_\beta$ and $\beta f - P(\beta f)$ are cohomologous and $L_{g_\beta}1 = 1$. The eigenmeasure $\mu_\beta$ of $L_{g_\beta}$ is $\sigma-$invariant and coincides with the equilibrium measure of $\beta f$, that is,
\[ \int \beta f\, d\mu_\beta + h(\mu_\beta) = P(\beta f)=\sup_{\{\mu \in P(X); \,\mu \,\text{is}\,\,\sigma- \text{invariant }\}} \int \beta f \,d\mu +h(\mu).  \]   Furthermore $d\mu_\beta = h_\beta\, d\nu_\beta$.
Classical results on thermodynamic formalism can be found in \cite{Bowen} and \cite{PP}.

 At the zero temperature case, in thermodynamic formalism, the above objects are studied for large $\beta$. In this case some intersections with ergodic optimization appear (\cite{BLL}, \cite{CLT}, \cite{CG}, \cite{Jenkinson}). It is well known, for instance, that $\lim_{\beta \to +\infty}\frac{P(\beta f)}{\beta} = m(f)$, where
\begin{equation}\label{eq10}
m(f) := \sup_{\{\mu \in P(X);\,\mu\,\text{is}\, \sigma-\text{invariant}\}} \int f\, d\mu.
\end{equation}
Any possible limit (weak* topology) of a subsequence of $(\mu_\beta)_{\beta >0}$  attains the supremum in (\ref{eq10}) being called a maximizing measure of $f$. Furthermore, $(\frac{1}{\beta}\log(h_\beta))_{\beta >0}$ is an equicontinuous family and any possible uniform limit $V$ of a subsequence of $(\frac{1}{\beta}\log(h_{\beta}))$ is a calibrated subaction \cite{CLT}, that is, it satisfies, for any $x \in X$, the equation 
\[ \sup_{\sigma(y)=x} [f(y) + V(y) - V(x) - m(f)] =0. \]
The limit function $V$ is Lipschitz and $R_{-} := f + V - V\circ\sigma - m(f)$ is the uniform limit of the corresponding subsequence of $\frac{g_{\beta}}{\beta}$, which satisfies: \newline
1) $R_{-}$ is Lipschitz and $R_{-} \leq 0$, \newline
2) $R_{-}$ and $f-m(f)$ are cohomologous, \newline
3) For any $x\in X$ there exists $y\in \sigma^{-1}(x)$ satisfying $R_{-}(y) =0$.\newline
Define $R_{+}:= -R_{-}$, $R_{+}^{n}(x) := \sum_{j=0}^{n-1}R_{+}(\sigma^{j}(x))$ and $R_{+}^{\infty}(x):=\lim_{n\to+\infty}R_{+}^{n}(x)$ ($R_+^\infty$ can assume the value $+\infty$).

Subactions and maximizing measures are dual objects linked in a particular form when we study the speed of convergence of $\mu_\beta$ to a maximizing measure.
From \cite{BLT} is known that, when the maximizing measure of $f$ 
is unique, there exists the uniform limit $R_{-}$ of $\frac{g_\beta}{\beta}$,\,\, ($\beta \to +\infty$). Furthermore, the measures $(\mu_{\beta})_{\beta>0}$ 
satisfy a Large Deviation Principle (L.D.P.), in the following sense, also used in the present work: there exists a lower semi-continuous function $I:X\to [0,+\infty]$ satisfying, for any cylinder  $k \subset X$,
\[\lim_{\beta\to+\infty} \frac{1}{\beta}\log(\mu_{\beta}(k)) = - \inf_{x\in k} I(x).\]
The deviation function $I$ in \cite{BLT} is given by $I=R_{+}^{\infty}$. It can assume the value $+\infty$.

In  \cite{LM2, M} this result has been generalized. Given $x \in X$, $n\in \mathbb{N}$ and $\beta>0$, consider the probability $m_{x,\beta,n} \in P(X)$ defined by $\int w \,dm_{x,\beta,n}= L_{g_{\beta}}^{n}(w)(x)$. If the maximizing measure of $f$ is unique, then\footnote{ we write $\lim_{n,\beta \to +\infty} a_{\beta,n}= a$ if for any $\epsilon>0$ there exists $L>0$ such that  $n,\beta >L \Rightarrow |a_{\beta,n} -a|<\epsilon$.}      
\begin{equation}\label{eq9}
\lim_{n,\beta\to +\infty}\frac{1}{\beta}\log(m_{x,\beta,n} (k))= \lim_{n,\beta\to +\infty}\frac{1}{\beta}\log( L_{g_{\beta}}^{n}(\chi_{k})(x)) = - \inf_{z\in k} R_{+}^{\infty}(z)
\end{equation} 
for any cylinder $k \subset X$.
 Given a continuous function $w$, $L_{g_{\beta}}^{n}(w)$ converges uniformly to $\int w\, d\mu_\beta$ ($n\to+\infty$). Therefore, for any $x \in X$, the probabilities $\mu_{x,\beta,n}$ converge to $\mu_\beta$ in the weak* topology ($n\to +\infty$). Consequently, from (\ref{eq9}), for any $x \in X$, 
\[\lim_{\beta \to +\infty}\frac{1}{\beta}\log(\mu_\beta (k)) =\lim_{\beta \to +\infty}\lim_{n\to+\infty}\frac{1}{\beta}\log(m_{x,\beta,n} (k)) = - \inf_{z\in k} R_{+}^{\infty}(z).   \]  

In \cite{BMP} the main result of \cite{BLT}, stated above, was proved for a more general class of functions (satisfying the Walters condition) on a countable mixing subshift with the BIP property. 
However, in both works it was assumed the existence of a unique maximizing measure to $f$.

If we do not assume the hypothesis of unicity, then there are some natural questions to be considered: 

\bigskip

\noindent
\textbf{question 1:} there exists   $V:=\lim_{\beta\to+\infty} \frac{1}{\beta}\log(h_\beta)$? \newline
\textbf{question 2:} there exists  $\lim_{\beta\to+\infty} \frac{1}{\beta}\log(\mu_{\beta}(k))$ for any cylinder $k$? \newline
\textbf{question 3:} there are relations between $V$ and the deviation function $I$?

\bigskip

Initially, it is natural to assume that the answer to the question 3 can be obtained generalizing the results in \cite{BLT}. In the case of existing the uniform limit $R_{+}=\lim_{\beta\to+\infty} -g_{\beta}/\beta$,
we could try to prove that $\mu_\beta$ satisfies a L.D.P. with deviation function $I=R_{+}^{\infty}$. However, in \cite{BLM} it is proved that this assertion is false. Even in the case there exist the limits in questions 1 and 2, one can get an explicit example where $I \neq R_{+}^{\infty}$.

We will show that, when the assertion in question 2 is satisfied, an affirmative answer to the question $1$ exists. In this case, we also present an answer to the question $3$, determining relations between $R_+$ and $I$.
Several results are known concerning the problem of selection of a maximizing measure \cite{BLM, BGT, Bremont,  Chazottes, CL, Kempton, Leplaideur, Leplaideur2}. In this work we present an improvement in the study of selection of the subaction as a consequence of our results on large deviations on the first part of the paper.

Define
\[M_{max}(f):= \{\mu\,\,\in\,P(X): \mu \,\,\text{is}\,\, \sigma-\text{invariant} \,\,\text{and}\,\, \int f\,d\mu = m(f)\}\]
and
\[X_{max}(f):= \bigcup_{\mu\in M_{max}(f)} \operatorname{supp}(\mu).\]
$X_{max}(f)$ is called the Mather set of $f$. It is a subset of the Aubri set
\small
\[\Omega(f)=\left\{x\in X\,|\, \lim_{\epsilon \rightarrow 0^+}\sup_{n\geq 1 }\sup_{\substack{d(x,z)<\epsilon\\ \sigma^n(z)=x}}[(f(z)+...+f(\sigma^{n-1}z)-n\cdot m(f)] =0\right\}.\]
\normalsize
Furthermore,  $\operatorname{supp}(\mu)\subset X_{max}(f)$ iff $\operatorname{supp}(\mu)\subset \Omega(f)$ iff  $\mu$ is a maximizing measure of $f$ \cite{CLT}.
 
\vspace{0.5cm}

In the section 2. we will prove the following theorem:

\begin{theorem}\label{teorema2}
With the above notations, suppose that for any cylinder $k\subset X$, there exists $\displaystyle{\lim_{\beta\to+\infty} \frac{1}{\beta}\log(\mu_{\beta}(k))}.$ Then,  denoting $x=(x_{1}x_{2}x_3...)$, \newline

\noindent
1. The family of probabilities $(\mu_\beta)_{\beta >0}$ satisfies a L.D.P. with deviation function  $I:X\to[0,+\infty]$,
\begin{equation} \label{eq6} 
I(x):= -\lim_{n\to+\infty}\lim_{\beta\to+\infty}\frac{1}{\beta}\log(\mu_{\beta}([x_{1}...x_{n}])).
\end{equation} 

\noindent 
2. There exists the uniform limit $R_{+}:= - \lim_{\beta\to+\infty} \frac{g_{\beta}}{\beta}$. It satisfies
 \[I= R_{+} + I\circ\sigma \,\,\,\,\, \text{and} \,\,\,\,\, I \geq R_{+}^{\infty} .\] 

\noindent
3. \begin{equation}\label{eq7}
I(x) = \inf_{y\in X_{max}(f)} \liminf_{n\to+\infty}\left(R_{+}^{n}(x_{1}...x_{n}y) + I(y)\right).
\end{equation}  
If $R_{+}^{\infty}(x)<+\infty$, there exists at least one point $y \in X_{max}(f)$ which is an accumulation point of $\{\sigma^n x\}_{n=1,2,...}$. For any such $y$ 
\[I(x) = R_{+}^{\infty}(x) + I(y).\]

\noindent 
4. There exists the uniform limit $V=\lim_{\beta \to +\infty}\frac{1}{\beta}\log(h_{\beta})$. \newline

\noindent
5. The eigenmeasures $\nu_{\beta}$ satisfy a L.D.P. with deviation function $I+V$.

\end{theorem}

\bigskip

Some remarks:

1. In the equation (\ref{eq6}) we do not exclude the possibility $I(x)=+\infty$. When we write $I(x)=R_{+}(x)+ I(\sigma x)$ we may have $+\infty = R_{+}(x)+\infty$. Following the above discussion, the function $R_{+}$ is real-valued, nonnegative, Lipschitz and for any $x\in X$, $\min_{a \in \{1,...,d\} }R_{+}(ax)=0$. 
\bigskip

2. Under the hypothesis of the theorem we get:

2.1. There exists at least one point $\tilde{y} \in X_{max}(f)$ satisfying $I(\tilde{y})=0$. Indeed, let $\mu_\infty$ be a probability on $X$ such that, for an increasing sequence  $\beta_i\to +\infty$,  $\mu_{\beta_i}\to \mu_\infty$ (weak* topology). Let $\tilde{y} \in \operatorname{supp}(\mu_\infty)$, that is, $\mu_\infty([y_1...y_n])>0$ for any cylinder $[y_1...y_n]$ containing $\tilde{y}$. In this way, from the hypothesis of the theorem,  
$$\lim_{\beta \to +\infty}\frac{1}{\beta}\log(\mu_\beta([y_1...y_n]))= \lim_{\beta_i \to +\infty}\frac{1}{\beta_i}\log(\mu_{\beta_i}([y_1...y_n])) = 0,$$
because $\mu_{\beta_i}([y_1...y_n])\to \mu_{\infty}([y_1...y_n])>0$. 
It follows from item 1. of the theorem that $I(\tilde{y})=0$. 

\bigskip

2.2. If $R_{+}^\infty(x)<+\infty$ then $I(x)<+\infty$. Consequently, as  $R_{+}^\infty(x)=0$ for any $x\in X_{max}(f)$, we conclude that $I(x)<+\infty$ for any $x \in X_{max}(f)$.  

Indeed, in the proof of the Theorem \ref{teorema2} we will show that (see eq. (\ref{eq11}) below)
\[I(x) \leq \liminf_{n\to +\infty}R_{+}^{n}(x_{1}...x_{n}y) + I(y) \,\,\,\, \forall\, y\in X. \]
 As $I(\tilde{y})=0$ at some point $\tilde{y} \in X_{max}(f)$ and $R_{+}$ is Lipschitz, there exists a constant $c>0$ satisfying
\[I(x) \leq \liminf_{n\to +\infty} R_{+}^{n}(x_{1}...x_{n}\tilde{y}) \leq \liminf_{n\to+\infty}R_{+}^{n}(x) +c(\theta+...+\theta^n) \leq R_{+}^{\infty}(x) + \frac{c\theta}{1-\theta} . \]

\bigskip

2.3. In the equation (\ref{eq7}), if $R_+^\infty(x)=+\infty$ then, following computations as above, we get $$\liminf_{n\to+\infty}\left(R_{+}^{n}(x_{1}...x_{n}y) + I(y)\right)=+\infty\,\,\, \forall y \in X_{max}(f).$$ In this case we write
\[ \inf_{y\in X_{max}(f)} \liminf_{n\to+\infty}\left(R_{+}^{n}(x_{1}...x_{n}y) + I(y)\right) =+\infty. \] 
If $R_+^\infty(x)<+\infty$ then $$\liminf_{n\to+\infty}\left(R_{+}^{n}(x_{1}...x_{n}y) + I(y)\right)<+\infty\,\,\, \forall y \in X_{max}(f).$$ 

\bigskip

2.4. If $I(y)=0$ for all $y\in X_{max}(f)$ then $I(x) = R_{+}^{\infty}(x)$  for all $x \in X$ (it follows from item 3. of the theorem). This is the case, for instance, if the maximizing measure of $f$ is unique.

\bigskip

2.5. The equation (\ref{eq7}) remains valid if we replace $y\in X_{max}(f)$ by $ y \in X$. It follows a similar argument with $\inf_{y \in X_{max}(f)}$ replaced by $\inf_{y \in X}$ in the proof.
\bigskip

3. There exist constants $C_1,C_2>0$ satisfying, for any $x,y \in X$, $n\geq 1$, and $\beta$ sufficiently large, 
\begin{equation}\label{eq12}
-\beta C_1 < \log(h_\beta(x)) < \beta C_1,\,\, |\log (h_\beta(x))-\log(h_\beta(y))|<\beta C_1 d_\theta(x,y)
\end{equation}
and
\begin{equation}\label{eq13}
e^{-\beta n C_2} <\mu_{\beta}([x_1...x_n])<e^{\beta n C_2}.
\end{equation}
For a proof of (\ref{eq12}), see \cite{CLT} p. 1404 or \cite{M} Lemma 28. For a proof of $(\ref{eq13})$, see \cite{PP}, proof of the corollary 3.2.1., observing  that $\mu_\beta$ is the equilibrium measure of $\beta f+ \log(h_\beta) - \log(h_\beta \circ \sigma) - P(\beta f)$, and use $(\ref{eq12})$.  

If the hypothesis of the theorem is not satisfied, from (\ref{eq13}), applying a Cantor's diagonal argument, we obtain the existence of a sequence $\beta_{j}$ for which all limits $\lim_{\beta_j \to +\infty}\frac{1}{\beta_j}\log(\mu_{\beta_j}(k))$ exist. A similar result is valid for this subsequence, with all $\beta$ replaced by $\beta_j$ in the statement of the theorem. 

\bigskip
   
4. If $X$ is a subshift of finite type defined from an aperiodic matrix, the theorem remains valid except by the equation $(\ref{eq7})$ which must be replaced by the following equation
\begin{equation}\label{eq16}
I(x) = \inf_{y \in X_{max}(f)} \left[\lim_{\epsilon \to 0^+}\inf_{n\geq 1}\inf_{\substack{d(x,z)<\epsilon\\ \sigma^n(z)=y}} R_{+}^n(z)+I(y)\right].
\end{equation}
We will prove the equation (\ref{eq16}) after the proof of the Theorem \ref{teorema2}. 

\vspace{0.5cm}

	In the section 3. we will apply the above theorem studying the L.D.P. for the equilibrium measures of a class of Lipschitz functions $f:\{0,1\}^\mathbb{N}\to\mathbb{R}$ satisfying\footnote{we use the following notations
		\[ 0^2=00, \, 0^3=000,...,\,0^\infty = (0000...), \,\, 1^2=11,\,1^3=111,..., \,1^\infty=(1111...). \]}
\[f|_{[01]}=b,\,\,\,f|_{[10]}=d,\,\,\,\,f(0^{\infty})=f(1^{\infty})=0,\,\,\,  f|_{[0^n1]} =a_n,\,\,\,\,f|_{[1^n0]} =c_n,\,\,n\geq 2             \] 
where $b,d, a_n, c_n<0$.   
The deviation function $I$ is presented and, in the case $\sum_{j\geq2} a_j < b+d+\sum_{j\geq 2} c_j$, this function differs from the one that appears in \cite{BLT} (see also \cite{BLM}).

\section{Proof of theorem $\ref{teorema2}$}

The following general result  is very helpful and proves item 1. of Theorem \ref{teorema2}.

\begin{lemma}\label{teorema3}
Let $\eta_{\beta}$ be a sequence of probabilities on $X$. Suppose  that for any cylinder $k\subset X$ there exists the limit $\displaystyle{\lim_{\beta\to+\infty} \frac{1}{\beta}\log(\eta_{\beta}(k))}$.
Then, denoting $x=(x_{1}x_{2}x_3...)$, \newline
1. The function $I:X\to [0,+\infty]$,
\[I(x):= -\lim_{n\to+\infty}\lim_{\beta\to+\infty}\frac{1}{\beta}\log(\eta_{\beta}([x_{1}...x_{n}]))\]
is lower semi-continuous.\newline
2. For any cylinder $k \subset X$, \[\lim_{\beta\to+\infty} \frac{1}{\beta}\log(\eta_{\beta}(k)) = -\inf_{x\in k} I(x).\]
\end{lemma}

\noindent
\textbf{Remark:} In \cite{BMP} this result is generalized for Gibbs measures when considering a countable mixing subshift with the BIP property. 

\begin{proof} The function $I$ is well defined because $\psi_x(n):=\lim_{\beta\to+\infty}\frac{1}{\beta}\log(\eta_{\beta}([x_{1}...x_{n}]))$ exists and it is not increasing with $n$. The function $I:X\to [0,+\infty]$,  $I(x)=-\lim_{n\to+\infty}\psi_x(n)$ assume the value $+\infty$ if $\lim_{n\to+\infty} \psi_x(n)=-\infty$. Furthermore, denoting $z^n=(z^n_1z^n_2z^n_3...) \in X$,
	\begin{align*}
	I(x) &=\liminf_{n\to+\infty}[-\lim_{\beta\to+\infty}\frac{1}{\beta}\log(\eta_{\beta}([x_{1}...x_{n}]))]\\
	     &=\liminf_{n\to+\infty}\inf_{z^n \in [x_1...x_n]}[-\lim_{m\to+\infty}\lim_{\beta\to+\infty}\frac{1}{\beta}\log(\eta_{\beta}([x_{1}...x_{n}]))]\\
	     &\leq \liminf_{n\to+\infty}\inf_{z^n \in [x_1...x_n]}[-\lim_{m\to+\infty}\lim_{\beta\to+\infty}\frac{1}{\beta}\log(\eta_{\beta}([z^n_{1}...z^n_{m}]))]\\
	     &=\liminf_{n\to+\infty}\inf_{z^n \in [x_1...x_n]} I(z^n),	     
	     \end{align*}
therefore $I$ is lower semi-continuous.

Given a cylinder $k=[x_1...x_n]$, for any $z=(z_1z_2z_3...) \in k$ we have
\[\lim_{\beta\to+\infty} \frac{1}{\beta}\log(\eta_{\beta}(k)) \geq \lim_{m\to+\infty} \lim_{\beta\to+\infty} \frac{1}{\beta}\log(\eta_{\beta}([z_1...z_m]))=-I(z).  \] 
Thus, we get \[\lim_{\beta\to+\infty} \frac{1}{\beta}\log(\eta_{\beta}(k)) \geq \sup_{z\in k} -I(z)= -\inf_{z\in k} I(z).\]
On the other hand, as 
\begin{align*}
\lim_{\beta\to+\infty} \frac{1}{\beta}\log(\eta_{\beta}([x_1...x_n])) &=\lim_{\beta\to+\infty} \frac{1}{\beta}\log(\sum_{j=1}^d\eta_{\beta}([x_1...x_nj]))\\
&=\max_{j\in\{1,...,d\}}\lim_{\beta\to+\infty} \frac{1}{\beta}\log(\eta_{\beta}([x_1...x_nj])) , 
\end{align*} 
there exists $y=(y_1y_2y_3...) \in X$ satisfying
\begin{align*}
\lim_{\beta\to+\infty} \frac{1}{\beta}\log(\eta_{\beta}([x_1...x_n]))&=\lim_{\beta\to+\infty} \frac{1}{\beta}\log(\eta_{\beta}([x_1...x_ny_1]))\\
&=\lim_{\beta\to+\infty} \frac{1}{\beta}\log(\eta_{\beta}([x_1...x_ny_1y_2]))=.... 
\end{align*}
Therefore, we finally get
\[\lim_{\beta\to+\infty} \frac{1}{\beta}\log(\eta_{\beta}(k))=-I(x_1...x_ny) \leq  \sup_{z\in k} -I(z)= -\inf_{z\in k} I(z).\]
\end{proof}

\begin{lemma} Under the hypotheses of the Theorem \ref{teorema2}
the deviation function $I$ in (\ref{eq6}) satisfies $I(x)\geq I(\sigma(x))\,\,\forall x\in X$. Particularly, the function \begin{equation}\label{eq8}
I_0:X\to[0,+\infty],\,\,\, I_{0}(x):=\lim_{n\to+\infty} I(\sigma^{n}(x))
\end{equation}
 is constant on each orbit $\Omega_x=\{\sigma^n(x)\,|\, n\in\{0,1,2,3,...\}\},\,\,x\in X$. 
\end{lemma}

\begin{proof}
Denoting $x=(x_1x_2x_3...)$,  as $\mu_\beta$ is $\sigma-$invariant, for $n\geq 2$,
\[ \frac{1}{\beta}\log(\mu_{\beta}([x_{1}...x_{n}])) \leq  \frac{1}{\beta}\log(\sum_{j=1}^d\mu_{\beta}([jx_{2}...x_{n}]))= \frac{1}{\beta}\log(\mu_{\beta}([x_{2}...x_{n}]))
.\]
Then
\begin{align*}
I(x) &= -\lim_{n\to+\infty}\lim_{\beta\to+\infty}\frac{1}{\beta}\log(\mu_{\beta}([x_{1}...x_{n}])) \\
&\geq -\lim_{n\to+\infty}\lim_{\beta\to+\infty}\frac{1}{\beta}\log(\mu_{\beta}([x_{2}...x_{n}]))=I(\sigma(x)).
\end{align*}
\end{proof}

We write $y \in \omega(x)$, $x,y \in X$, if there exists an increasing sequence $n_i \to +\infty$ such that $\sigma^{n_i}(x)\to y$.

\begin{corollary}\label{omega}
Under the hypotheses  of Theorem \ref{teorema2}, let $I$ be the deviation function defined in (\ref{eq6}) and $I_0$ be the function defined in (\ref{eq8}). \newline
1. If $y\in \omega(x)$, then $I(y) \leq I_{0}(x)\leq I(x)$,\newline
2.  $I$ is constant on each periodic orbit, \newline
3. If $x\in \omega(y)$ and $y\in\omega(x)$, then $I_{0}(x)=I(x)=I(y)=I_{0}(y)$.
\end{corollary}

\begin{proof}
In order to prove 1. we suppose $\sigma^{n_{j}}(x) \to y$. From the above lemma and using the lower semi-continuity of $I$ we get
\[ I(x)\geq I_{0}(x) = \lim_{n_{j}\to+\infty}I(\sigma^{n_{j}}(x))\geq I(y).\]
The proof of 2. consists in observing that for a periodic orbit $\{x,...,\sigma^{n}(x)\}$ we have:
\[I(x) \geq I(\sigma(x)) \geq ... \geq I(\sigma^n(x)) \geq I(x).\]
Analogously,  to prove 3. we observe that from 1. we have
\[I(x) \geq I_{0}(x) \geq I(y) \geq I_{0}(y) \geq I(x).\]
\end{proof}

\bigskip

\noindent
\textbf{Proof of Theorem \ref{teorema2}:}\newline

\noindent 
\textit{proof of 1.} \newline
It is a consequence of Lemma \ref{teorema3}.

\bigskip

\noindent
\textit{proof of 2.}\newline
The existence of the limit $R_{-}$ is a consequence of corollary 48 in \cite{M}. Furthermore, following the Proposition 47 in \cite{M},  for $x=(x_{0}x_{1}x_2...)$, we get
\begin{align*}
R_{+}(x) &= \lim_{\beta,n\to+\infty} \frac{1}{\beta}\log \frac{\mu_{\beta}[x_{1}...x_{n}]}{\mu_{\beta}[x_{0}...x_{n}]}\\ 
&= \lim_{n\to+\infty}\left(\lim_{\beta\to+\infty}\frac{1}{\beta}\log \mu_{\beta}[x_{1}...x_{n}] - \lim_{\beta\to+\infty}\frac{1}{\beta}\log \mu_{\beta}[x_{0}...x_{n}]\right).
\end{align*}
Therefore, using Lemma \ref{teorema3}, we have
\[I(x) = R_{+}(x) + I(\sigma(x)).\]
It follows that for each $n$,
\begin{equation}\label{IandR}
I(x) = R_{+}^{n}(x) + I(\sigma^{n}(x)),
\end{equation}
and, taking $n\to+\infty$,
\[I(x) = R_{+}^{\infty}(x) + I_{0}(x)\]
(see also  \cite{BMP}).
\bigskip

\noindent
\textit{proof of 3.}  \newline
Denoting $x=(x_{1}x_{2}x_3...)$, we want to show that 
$$I(x) = \inf_{y\in X_{max}(f)} \liminf_{n\to+\infty}\left(R_{+}^{n}(x_{1}...x_{n}y) + I(y)\right).$$
For a fixed $y\in X$ we have from (\ref{IandR}) that
\[I(x_{1}...x_{n}y) = R_{+}^{n}(x_{1}...x_{n}y) + I(y).\]
As $I$ is lower semi-continuous,
\begin{equation}\label{eq11}
I(x) \leq \liminf_{n\to+\infty} \left(R_{+}^{n}(x_{1}...x_{n}y) + I(y)\right).
\end{equation}
Then, (considering an infimum on the right side)
\[I(x) \leq \inf_{y\in X_{max}(f)} \liminf_{n\to+\infty}\left(R_{+}^{n}(x_{1}...x_{n}y) + I(y)\right).\]

Now, we will prove the reverse inequality.

If $R_{+}^{\infty}(x)=+\infty$, then $I(x)= R_{+}^{\infty}(x) +I_0(x)=+\infty$ and, as $R_{+}$ is a Lipschitz function, for any $y \in X_{max}(f)$, 
\[\liminf_{n\to+\infty}R_{+}^{n}(x_{1}...x_{n}y) =+\infty.\]
So  the main equality (\ref{eq7}) holds. 

If $R_{+}^{\infty}(x)<+\infty$, there exists at least one point $y \in X_{max}(f)$ such that is an accumulation point of the sequence $\{\sigma^{n}(x)\}_{n=0,1,...}$ (see \cite{LMST} or Lemma 42 and Cor. 43 in \cite{M} or \cite{BMP}). We write $y=(y_{1}y_{2}y_3...).$ \newline
It follows that for each $j\in \mathbb{N}$ there exists some $m_j>j$ such that 
\[x = (x_{1}...x_{m_j}\, y_{1}...y_{j}\,x_{m_{j}+j+1}x_{m_{j}+j+2}...).\]
Then 
\[I(x) = (R_{+}^{m_j}(x_{1}...x_{m_j}y_{1}...y_{j}x_{m_{j}+j+1}...) + I(y_{1}...y_{j}x_{m_{j}+j+1}...).\]
When $j\to+\infty$, using the fact that $I$ is lower semi-continuous and $R_{+}$ is Lipschitz, we get
\begin{align*}
I(x) &\geq \liminf_{j\to+\infty}(R_{+}^{m_j}(x_{1}...x_{m_j}y_{1}...y_{j}x_{m_{j}+j+1}...)) + I(y)\\ 
&= \liminf_{j\to+\infty}(R_{+}^{m_j}(x_{1}...x_{m_j}y)) + I(y)\\
&\geq \liminf_{n\to+\infty}(R_{+}^{n}(x_{1}...x_{n}y)) + I(y).
\end{align*}
Therefore,
\[I(x) \geq \inf_{y\in X_{max}(f)} \liminf_{n\to+\infty}\left(R_{+}^{n}(x_{1}...x_{n}y) + I(y)\right),\]
proving the reverse inequality. This concludes the proof of (\ref{eq7}). 

\bigskip
As we see above, if $R_{+}^{\infty}(x)<+\infty$, there exists at least one point $y \in X_{max}(f)$ which is an accumulation point of $\{\sigma^n x\}_{n=1,2,...}$. 
For any such $y$, from corollary \ref{omega}, we have
\[I(x) = R_{+}^{\infty}(x)+I_0(x) \geq R_{+}^{\infty}(x)+I(y).\]
On the other hand, following the notations above 
\begin{align*}
I(x) &\leq \liminf_{j\to+\infty} I(x_1...x_{m_j}y) \\
&= \liminf_{j\to+\infty}R_{+}^{m_j}(x_1...x_{m_j}y)+I(y)=\lim_{m_j\to+\infty}R_{+}^{m_j}(x)+I(y)=R_{+}^{\infty}(x)+I(y),
\end{align*}
where we use that $I$ is lower semi-continuous, $R_{+}^{n}(x)$ is increasing with $n$, $x_{m_j+1}...x_{m_j+j}=y_1...y_j$ and that $R_{+}$ is Lipschitz.
This concludes the proof of the equation
\[I(x) = R_{+}^{\infty}(x) + I(y).\]

\bigskip
\noindent
\textit{Proof of 4. and 5. } \newline
We denote by $\nu_{\beta}$ the eigenmeasure of the Ruelle Operator $L_{\beta f}$. The probabilities $\nu_{\beta}$ and  $\mu_{\beta}$ satisfy $h_\beta \,d\nu_{\beta}=d\mu_{\beta}$, that is,
\[\int w\cdot h_\beta \,d\nu_{\beta} = \int w\, d\mu_{\beta}\,\,\, \forall w \in C(X).\]
Consequently, given a cylinder $k \subset X$, from (\ref{eq12}) and (\ref{eq13}), there exists a constant $C_k$ such that, for $\beta$ sufficiently large, $-\beta C_k < \log(\nu_\beta(k)) < \beta C_k$.

We suppose initially the existence of the uniform limit $$V_{1}:=\lim_{\beta_j \to+\infty}\frac{1}{\beta_j}\log(h_{\beta_j}).$$ For a cylinder $k_{0}$ and an accumulation point $a$ of $\frac{1}{\beta_j}\log(\nu_{\beta_j}(k_{0}))$, there exists a subsequence  $\beta_{j_i}$ such that
\[\lim_{\beta_{j_i}\to+\infty}\frac{1}{\beta_{j_i}}\log(\nu_{\beta_{j_i}}(k_{0})) = a.\]
Using a Cantor's diagonal argument we can suppose that for any cylinder $k$ there exists the limit of $\frac{1}{\beta_{j_i}}\log(\nu_{\beta_{j_i}}(k))$.
By hypothesis,  $\mu_{\beta}$ satisfies a L.D.P. with deviation function $I$. Fixed any point $z=(x_{1}x_{2}x_3...)$, for each $n$,
\[
\frac{1}{\beta_{j_i}}\log (\nu_{\beta_{j_i}}([x_{1}...x_{n}])+\inf_{[x_{1}...x_{n}]}\frac{1}{\beta_{j_i}}\log(h_{\beta_{j_i}}) \leq 
 \frac{1}{\beta_{j_i}}\log(\mu_{\beta_{j_i}}([x_{1}...x_{n}]))  \]
\[\leq \frac{1}{\beta_{j_i}}\log (\nu_{\beta_{j_i}}([x_{1}...x_{n}])+\sup_{[x_{1}...x_{n}]}\frac{1}{\beta_{j_i}}\log(h_{\beta_{j_i}}).\]
Taking $\beta_{j_i}\to+\infty$, we have
\[
\lim_{\beta_{j_i}\to+\infty}\frac{1}{\beta_{j_i}}\log (\nu_{\beta_{j_i}}([x_{1}...x_{n}])+\inf_{[x_{1}...x_{n}]} V_{1} \leq \lim_{\beta_{j_i}\to+\infty} \frac{1}{\beta_{j_i}}\log(\mu_{\beta_{j_i}}([x_{1}...x_{n}]))\]
\[\leq \lim_{\beta_{j_i}\to+\infty}\frac{1}{\beta_{j_i}}\log (\nu_{\beta_{j_i}}([x_{1}...x_{n}])+\sup_{[x_{1}...x_{n}]} V_{1} .\]
When $n \to+\infty$    (applying Lemma \ref{teorema3}) we have
\[-\lim_{n\to+\infty}\lim_{\beta_{j_i}\to+\infty}\frac{1}{\beta_{j_i}}\log (\nu_{\beta_{j_i}}([x_{1}...x_{n}]) = I(z) + V_{1}(z).\]
Using Lemma $\ref{teorema3}$ again, we conclude that $\nu_{\beta_{j_i}}$ satisfies a L.D.P. with deviation function $I+V_{1}$. Then 
\[a = - \inf_{x\in k_{0}} (I(x) + V_{1}(x)).\]
As $a$ is any possible accumulation point of  $\frac{1}{\beta_j}\log(\nu_{\beta_j}(k_{0}))$ we conclude that 
\[\lim_{\beta_j \to \infty} \frac{1}{\beta_j}\log(\nu_{\beta_j}(k_{0}))=- \inf_{x\in k_{0}} (I(x) + V_{1}(x)).\]

Now we will prove the existence of the limit function $V$. 
Suppose that for subsequences $\beta_{i}$ and $\beta_{j}$ we have 
\[\lim_{\beta_i \to+\infty}\frac{1}{\beta_{i}}\log(h_{\beta_{i}}) = V_{1} \,\,\text{and}\,\, \lim_{\beta_j \to+\infty}\frac{1}{\beta_{j}}\log(h_{\beta_{j}}) = V_{2}.\]
 Applying 2. of the Theorem \ref{teorema2} we obtain $V_{2} -V_2\circ\sigma  = V_1 - V_{1}\circ\sigma$. Therefore, $V_2 = V_1 + C$ for some constant $C$.

Applying the above conclusions on the L.D.P. for the set $X$ (the full space) we get
\[0 = \lim_{\beta_{i}\to+\infty}\frac{1}{\beta_{i}}\log (\nu_{\beta_{i}}(X)) = -\inf_{x\in X}(I(x) + V_{1}(x))\]
and
\[0 = \lim_{\beta_{j}\to+\infty}\frac{1}{\beta_{j}}\log (\nu_{\beta_{j}}(X)) = -\inf_{x\in X}(I(x) + V_{2}(x)).\]
Thus, we have
\[0 = -\inf_{x\in X}(I(x) + V_{2}(x)) = -\inf_{x\in X}(I(x) + V_{1}(x) + C)\]
\[ = -\inf_{x\in X}(I(x) + V_{1}(x)) + C =0+C= C,\]
proving that $V_{2}=V_{1}$. This shows that exists the uniform limit $V=\lim_{\beta \to \infty} \frac{1}{\beta}\log(h_\beta)$, proving 4.

 The previous arguments on the L.D.P. for the measures $\nu_{\beta_j}$ can be applied to the general family of measures $\mu_\beta$, proving 5. $\square$

\bigskip

\noindent
\textbf{Proof of remark 4. and equation (\ref{eq16}):}
If $X$ is a subshift of finite type defined from an aperiodic matrix, the arguments used in the above proof are also valid except by  some estimates in the proof of (\ref{eq7}). 
In this case the equation (\ref{eq7}) can be replaced by the equation (\ref{eq16}), that is,
\[ I(x) = \inf_{y \in X_{max}(f)} \left[\lim_{\epsilon \to 0^+}\inf_{n\geq 1}\inf_{\substack{d(x,z)<\epsilon\\ \sigma^n(z)=y}} R_{+}^n(z)+I(y)\right].\]
Indeed, from (\ref{IandR}), 
\[ \inf_{\substack{d(x,z)<\epsilon\\ \sigma^n(z)=y}} R_{+}^n(z)+I(y) = \inf_{\substack{d(x,z)<\epsilon\\ \sigma^n(z)=y}} I(z).\]
As $I$ is lower semi-continuous, for any $x,y \in X$, we have
\[I(x)\leq \left[\lim_{\epsilon \to 0^+}\inf_{n\geq 1}\inf_{\substack{d(x,z)<\epsilon\\ \sigma^n(z)=y}} I(z)\right],\]
and, then
\[ I(x) \leq \inf_{y \in X_{max}(f)} \left[\lim_{\epsilon \to 0^+}\inf_{n\geq 1}\inf_{\substack{d(x,z)<\epsilon\\ \sigma^n(z)=y}} R_{+}^n(z)+I(y)\right].\]

In order to prove the reverse inequality we remark that if $R_+^\infty(x)=+\infty$, then $I(x)=R_+^\infty(x) +I_0(x)=+\infty$ and using the above inequality, the equation (\ref{eq16}) corresponds to the equality $+\infty=+\infty$.
Suppose $R_+^\infty(x)<\infty$ and consider $\eta >0$. As $R_+$ is Lipschitz, there exists a constant $C>0$ such that $|R_+(a)-R_+(b)| \leq C d_{\theta}(a,b),\,\,\forall a,b \in X$. Let $j_0$ be such that  $\frac{C \theta^{j_0}}{1-\theta} < \eta$.
Take $y\in X_{max}(f)\cap \omega(x)$. Then, from corollary \ref{omega},
\[I(x) = R_{+}^{\infty}(x)+I_0(x) \geq R_{+}^{\infty}(x)+I(y).\] 
 Given $\epsilon >0$, let $j>j_0$ be such that $\theta^j < \epsilon$. For this $j$ there exists $m_j>j$ such that
\[x = (x_{1}...x_{m_j}\, y_{1}...y_{j}\,x_{m_{j}+j+1}x_{m_{j}+j+2}...).\]
Let $z_\epsilon = (x_1...x_{m_j}y)$. 
Then,
\begin{align*}
I(x) &\geq  R_+^{\infty}(x) + I(y) \geq R_+^{m_j}(x) +I(y) \\
&\geq R_+^{m_j}(x_1...x_{m_j}y) +I(y) -\eta =R_+^{m_j}(z_\epsilon) +I(y) -\eta .
\end{align*}
Therefore, $d(x,z_\epsilon)<\epsilon$, $\sigma^{m_j}(z_\epsilon)=y$ and $I(x)\geq (R_{+}^{m_j}(z_{\epsilon}) + I(y)) -\eta$. This construction shows that
\[ I(x) \geq \inf_{y \in X_{max}(f)} \left[\lim_{\epsilon \to 0}\inf_{n\geq 1}\inf_{\substack{d(x,z)<\epsilon\\ \sigma^n(z)=y}} R_{+}^n(z)+I(y)\right] -\eta.\]
As $\eta$ can be arbitrarily small, we conclude the proof. $\square$

\section{Application for an explicit example}

Now we use the results described above in order to complete the study of Large Deviations for the equilibrium measures of a family of functions previously studied in \cite{BLM}.
\begin{definition} 
We write  $f\in \mathbf{W}$ if $f:\{0,1\}^\mathbb{N} \to \mathbb{R}$ is a Lipschitz function and there exist negative numbers $b,d,\{c_n\}_{n\geq 2},\{a_n\}_{n\geq 2}$, such that, for $n\geq 2$,
\begin{equation}\label{eq14}
f|_{[01]}=b,\,\,\,f|_{[10]}=d,\,\,\,\,f(0^{\infty})=f(1^{\infty})=0,\,\,\,  f|_{[0^n1]} =a_n,\,\,\,\,f|_{[1^n0]} =c_n.   
\end{equation}
\end{definition}

Any function $f \in \mathbf{W}$ belongs to the class of potentials defined by P. Walters \cite{Walters} where $0=a=c$, $b=b_1=b_2=...$ and $d=d_1=d_2=...$\,. We remark that $\sum_{i\geq 2}a_i >-\infty$ and $\sum_{i\geq 2} c_i >-\infty$, because  $f$ is Lipschitz and $f(0^{\infty})=f(1^{\infty})=0$.

In the analysis of the zero temperature case for these functions, the exponential limit of $P(\beta f)$ plays an important role.
\begin{lemma}\label{pressure} If $f\in \mathbf{W}$ satisfies (\ref{eq14}), then
	\[\lim_{\beta\to +\infty}\frac{1}{\beta}\log(P(\beta f)) =A\] where
	\[A= \left\{ \begin{array}{ll}
b+d +\sum_{j=1}^{\infty}c_{1+j}, & \text{when} \ \ \sum_{j=1}^{\infty}a_{1+j} \leq b+d +\sum_{j=1}^{\infty}c_{1+j},\\
b+d +\sum_{j=1}^{\infty}a_{1+j}, & \text{when} \ \ \sum_{j=1}^{\infty}c_{1+j} \leq b+d +\sum_{j=1}^{\infty}a_{1+j},\\
\frac{b+d}{2} +\sum_{j=1}^{\infty}\frac{a_{1+j}}{2} +\sum_{j=1}^{\infty}\frac{c_{1+j}}{2}, & \text{in the other cases}.
	\end{array} \right.     \]
\end{lemma}
\begin{proof}
	See Prop. 12 in \cite{BLM}.
\end{proof}

Given $f \in \mathbf{W}$ satisfying (\ref{eq14}) and $\beta >0$, in order to simplify the computations, we will consider the function $H_\beta(x)=\frac{h_{\beta}(x)}{h_\beta(0^\infty)}$. This normalization of the eigenfunction was used in \cite{BLM}. Observe that  $H_\beta(0^{\infty}) =1$ and  $\log(H_\beta)-\log(H_\beta\circ\sigma) = \log(h_\beta)-\log(h_\beta\circ\sigma)$. Therefore
\[g_\beta= \beta f+ \log(h_\beta)-\log(h_\beta\circ\sigma) -P(\beta f) = \beta f+ \log(H_\beta)-\log(H_\beta\circ\sigma) -P(\beta f).      \] 
Following \cite{Walters} (see Theo. 3.1 and page 1341), we obtain
\begin{align*} 
&H_{\beta}(0^{\infty}) = 1\\
&H_{\beta}(1^{\infty}) = \frac{e^{\beta b}}{e^{P(\beta f)}}\left(1+\sum_{j=1}^{\infty}e^{\beta(a_{2}+...+a_{1+j}) - jP(\beta f)}\right) \\
&H_{\beta}|_{[0^{q}1]} =  \frac{(e^{P(\beta f)}-1)}{e^{P(\beta f)}}\left(1+\sum_{j=1}^{\infty}e^{\beta(a_{q+1}+...+a_{q+j}) - jP(\beta f)}\right), \,\, q\geq 1\\
&H_{\beta}|_{[1^{q}0]} =  \frac{H_{\beta}(1^{\infty})(e^{P(\beta f)}-1)}{e^{P(\beta f)}}\left(1+\sum_{j=1}^{\infty}e^{\beta(c_{q+1}+...+c_{q+j}) - jP(\beta f)}\right),\,\, q\geq 1.
\end{align*}

The next lemma can be used in order to get the function $R_+^\infty$ that appears in the formulation of the deviation function in Theorem \ref{teorema2}.
 
\begin{lemma}\label{Vlimit} Under the above notations, for $f\in \mathbf{W}$ satisfying (\ref{eq14}), there exists the uniform limit $U=\lim_{\beta\to+\infty} \frac{1}{\beta}\log(H_{\beta})$. This function $U$ is a calibrated subaction for $f$ and  satisfies
	\begin{align*} 
&U(0^{\infty})=0,\\
&U(1^{\infty}) =  b   + \max\left\{0,  \sum_{j=1}^{\infty}(a_{1+j}) -A \right\},\\
&U|_{[0^{q}1]} =  A   +\max\left\{0, \sum_{j=1}^{\infty}(a_{q+j}) -A\right\},\,\, q\geq 1,\\
&U|_{[1^{q}0]} =   b +A  + \max\left\{0,  \sum_{j=1}^{\infty}(a_{1+j}) -A\right\} + \max\left\{0, \sum_{j=1}^{\infty}(c_{q+j}) -A\right\}, \,\, q\geq 1.
\end{align*}	
\end{lemma}

\begin{proof}
The result can be obtained as a particular case of Prop. 2 in \cite{BLM}.
\end{proof}

\noindent
\textbf{Remark:} if $V=\lim_{\beta \to +\infty}\frac{1}{\beta}\log(h_{\beta})$ then $U= V-C$, with $C=V(0^\infty)$. 

\bigskip

We want to study the L.D.P. for the equilibrium measures $\mu_{\beta}$ and naturally any maximizing measure of $f\in \mathbf{W}$ is a convex combination of the ergodic measures supported in the periodic orbits $0^{\infty}=(000...)$ and $1^{\infty}=(111...)$. From the above lemma there exists the limit function $$R_{+}=- \lim_{\beta\to+\infty}\frac{g_\beta}{\beta}=-f-U+U\circ\sigma +m(f)=-f-U+U\circ\sigma.$$
We want to use Theorem \ref{teorema2}. In order to do that we need to find the expression of deviation function. More precisely, we need to compute $I(0^\infty)$ and $I(1^{\infty})$. In this way, considering the Lemma \ref{teorema3}, we first study $\mu_\beta([0^n])$ and $\mu_\beta([1^{n}])$.

\begin{lemma}\label{mu0n} Let $f\in \mathbf{W}$ satisfying (\ref{eq14}). Then, for $\beta >0$ and $n\geq 1$,
\begin{equation}\label{eq5}	
 \mu_\beta([0^n])= \frac{S_0^n(\beta)}{S_0(\beta)+S_1(\beta)}\,\,\,\,\text{and}\,\,\,\, \mu_\beta([1^n])= \frac{S_1^n(\beta)}{S_0(\beta)+S_1(\beta)},  
\end{equation}
where 
\begin{align*}
&S_0(\beta)=S_0^1(\beta):= \frac{ 1 + \sum_{j=1}^{\infty}(j+1)e^{\beta(a_{2}+...+a_{1+j})-jP(\beta f)}}{1+\sum_{j=1}^{\infty}e^{\beta(a_{2}+...+a_{1+j})-jP(\beta f)}},\\ \\
&S_{1}(\beta)=S_1^1(\beta) := \frac{ 1+\sum_{j=1}^{\infty}(j+1)e^{\beta(c_{2}+...+c_{1+j})-jP(\beta f)}}{ 1+\sum_{j=1}^{\infty}e^{\beta(c_{2}+...+c_{1+j})-jP(\beta f)}},
\end{align*}
and for $n\geq 2$
\begin{align*} 
&S_0^n(\beta):= \frac{e^{P(\beta f)}\sum_{j=n}^\infty (j-n+1) e^{\beta(a_{2}+...+a_{j}) - jP(\beta f)} }{\left(1+\sum_{i=1}^{\infty}e^{\beta(a_{2}+...+a_{1+i}) -iP(\beta f)}\right)},\\ \\
&S_1^n(\beta):= \frac{e^{P(\beta f)}\sum_{j=n}^\infty (j-n+1) e^{\beta(c_{2}+...+c_{j}) - jP(\beta f)} }{\left(1+\sum_{i=1}^{\infty}e^{\beta(c_{2}+...+c_{1+i}) -iP(\beta f)}\right)}.
\end{align*} 
\end{lemma}

\begin{proof}
Following \cite{BLM}, page 1351,
\begin{align*}
 S_0(\beta) &:= \frac{ 1 + \sum_{j=1}^{\infty}(j+1)e^{\beta(a_{2}+...+a_{1+j})-jP(\beta f)}}{1+\sum_{j=1}^{\infty}e^{\beta(a_{2}+...+a_{1+j})-jP(\beta f)}}\\
 &=1+\sum_{j=2}^\infty  e^{\beta (a_2+...+a_j) + \log(H_{\beta}|_{[0^j1]})- \log(H_{\beta}|_{[01]})-(j-1)P(\beta f)}
\end{align*}	
and
\begin{align*}
S_{1}(\beta) &:= \frac{ 1+\sum_{j=1}^{\infty}(j+1)e^{\beta(c_{2}+...+c_{1+j})-jP(\beta f)}}{ 1+\sum_{j=1}^{\infty}e^{\beta(c_{2}+...+c_{1+j})-jP(\beta f)}}\\
&=1+\sum_{j=2}^\infty  e^{\beta (c_2+...+c_j) + \log(H_{\beta}|_{[1^j0]})- \log(H_{\beta}|_{[10]})-(j-1)P(\beta f)}.
\end{align*}	
For $j \geq 2$, (see page 1352 in \cite{BLM})
\begin{equation}\label{eq1}
\mu_\beta([0^{j}1])=\mu_{\beta}([01])e^{\beta (a_2+...+a_{j}) +\log(H_{\beta}|_{[0^j1]}) -\log(H_{\beta}|_{[01]})  - (j-1) P(\beta f)}
\end{equation} 
and 
\begin{equation}\label{eq2}
\mu_\beta([1^{j}0])=\mu_{\beta}([10])e^{\beta (c_2+...+c_{j}) +\log(H_{\beta}|_{[1^j0]}) -\log(H_{\beta}|_{[10]})  - (j-1) P(\beta f)}.
\end{equation} 
Then 
\[ \mu_\beta([0])= \sum_{j=1}^{\infty}\mu_{\beta}([0^j1])= \mu_{\beta}([01])S_0(\beta),  \]
and
\[ \mu_\beta([1])= \sum_{j=1}^{\infty}\mu_{\beta}([1^j0])= \mu_{\beta}([10]) S_1(\beta). \]
As $\mu_\beta([01])=\mu_\beta([10])$ (because $\mu_\beta$ is $\sigma-$invariant) and $\mu_\beta([0])+\mu_\beta([1])=1$ we obtain

\begin{equation}\label{eq3}
\mu_\beta([01])=\mu_\beta([10])=\frac{1}{S_0(\beta)+S_1(\beta)}.
\end{equation}
As a consequence,
\[ \mu_\beta([0])= \frac{S_0(\beta)}{S_0(\beta)+S_1(\beta)}  \,\,\, and \,\,\, \mu_\beta([1])= \frac{S_1(\beta)}{S_0(\beta)+S_1(\beta)}. \]

From (\ref{eq1}) and (\ref{eq3}), for any $n\geq 2$,
\begin{align*}
\mu_\beta([0^n]) &= \sum_{j=n}^{\infty}\mu_{\beta}([0^j1])\\
&= \mu_{\beta}([01])\sum_{j=n}^\infty  e^{\beta (a_2+...+a_j) + \log(H_{\beta}|_{[0^j1]})- \log(H_{\beta}|_{[01]})-(j-1)P(\beta f)}\\
&= \frac{\sum_{j=n}^\infty  e^{\beta (a_2+...+a_j) + \log(H_{\beta}|_{[0^j1]})- \log(H_{\beta}|_{[01]})-(j-1)P(\beta f)}}{S_0(\beta)+S_1(\beta)} .
\end{align*}

Furthermore, 
\begin{align*} \sum_{j=n}^\infty&  e^{\beta (a_2+...+a_j) + \log(H_{\beta}|_{[0^j1]})- \log(H_{\beta}|_{[01]})-(j-1)P(\beta f)}  \\ 
&= \sum_{j=n}^\infty  \frac{e^{\beta (a_2+...+a_j) -(j-1)P(\beta f)} H_{\beta}|_{[0^j1]} }{H_{\beta}|_{[01]}}        \\
&= \sum_{j=n}^\infty  \frac{e^{\beta (a_2+...+a_j) -(j-1)P(\beta f)} \left(1+\sum_{i=1}^{\infty}e^{\beta(a_{j+1}+...+a_{j+i}) - iP(\beta f)}\right) }{\left(1+\sum_{i=1}^{\infty}e^{\beta(a_{2}+...+a_{1+i}) -iP(\beta f)}\right)}        \\
&=  \frac{\sum_{j=n}^\infty\sum_{i=0}^{\infty}e^{\beta(a_{2}+...+a_{j+i}) - (j+i-1)P(\beta f)} }{\left(1+\sum_{i=1}^{\infty}e^{\beta(a_{2}+...+a_{1+i}) -iP(\beta f)}\right)}        \\
&= \frac{e^{P(\beta f)}\sum_{m=n}^\infty (m-n+1) e^{\beta(a_{2}+...+a_{m}) - mP(\beta f)} }{\left(1+\sum_{i=1}^{\infty}e^{\beta(a_{2}+...+a_{1+i}) -iP(\beta f)}\right)}  =S_0^n(\beta). 
\end{align*}
Therefore, we finally get
\[\mu_{\beta}[0^n]=  \frac{S_0^n(\beta)}{S_0(\beta)+S_1(\beta)}.
\]
The computation for $\mu_{\beta}[1^n]$ is similar.
\end{proof} 

\bigskip

As we want to determine the limit of $\frac{1}{\beta}\log(\mu_\beta([0^n]))$ and $\frac{1}{\beta}\log(\mu_\beta([1^n]))$ (see Lemma \ref{teorema3}) the next lemma is useful.

\begin{lemma}\label{limits} 
	Let $f\in \mathbf{W}$ satisfying (\ref{eq14}). Denote  $A= \lim_{\beta \to+\infty} \frac{1}{\beta} \log(P(\beta f)).$ Under the above notations, 
	\[ \lim_{\beta \to+\infty}\frac{1}{\beta}\log (S_0(\beta))= \max\{ 0, \sum_{j=2}^\infty a_{j} -2A\} - \max\{ 0, \sum_{j=2}^\infty a_{j} -A\}, \]
	\[ \lim_{\beta \to+\infty}\frac{1}{\beta}\log (S_1(\beta))= \max\{ 0, \sum_{j=2}^\infty c_{j} -2A\} - \max\{ 0, \sum_{j=2}^\infty c_{j} -A\} \]
	and for $n\geq 2$
	\[ \lim_{\beta \to+\infty}\frac{1}{\beta}\log (S_0^n(\beta))= \max\{ a_2+...+a_n, \sum_{j=2}^\infty a_{j} -2A\} - \max\{ 0, \sum_{j=2}^\infty a_{j} -A\}, \]
	\[ \lim_{\beta \to+\infty}\frac{1}{\beta}\log (S_1^n(\beta))= \max\{ c_2+...+c_n, \sum_{j=2}^\infty c_{j} -2A\} - \max\{ 0, \sum_{j=2}^\infty c_{j} -A\}. \]
	
\end{lemma}

\begin{proof}
	We only present the prove  of the first equation, because the arguments are similar for the other cases.
	Initially, observe that for any $j_1 \geq 0$,
	\[ \lim_{\beta \to+\infty} \frac{1}{\beta} \log\left( \sum_{j\geq j_1}(j+1)e^{-jP(\beta f)}  \right)= -2A  \]
	and
	\[  \lim_{\beta \to+\infty} \frac{1}{\beta} \log\left( \sum_{j\geq j_1}e^{-jP(\beta f)}  \right)= -A \]
	(see Cor. 14 in \cite{BLM}).

As 
\begin{align*} 
\frac{1}{\beta}\log (S_0(\beta))=&    \frac{1}{\beta}\log[ 1 + \sum_{j=1}^{\infty}(j+1)e^{\beta(a_{2}+...+a_{1+j})-jP(\beta f)}]\\
& -   \frac{1}{\beta}\log[1+\sum_{j=1}^{\infty}e^{\beta(a_{2}+...+a_{1+j})-jP(\beta f)}],
\end{align*}
we will study the limit for $\frac{1}{\beta}\log[ 1 + \sum_{j=1}^{\infty}(j+1)e^{\beta(a_{2}+...+a_{1+j})-jP(\beta f)}]$ and $  \frac{1}{\beta}\log[1+\sum_{j=1}^{\infty}e^{\beta(a_{2}+...+a_{1+j})-jP(\beta f)}]$.

As $a_i <0\,\, \forall i\in \{2,3,4,...\}$ we have,
\begin{align*}
  \liminf_{\beta \to+\infty}&\frac{1}{\beta}\log\left(1 + \sum_{j=1}^{\infty}(j+1)e^{\beta(a_{2}+...+a_{1+j})-jP(\beta f)}\right) \\
&\geq \max\left\{ 0\,,\,\, \liminf_{\beta \to+\infty}\frac{1}{\beta}\log\left( e^{\beta\sum_{i\geq 2 }a_{i}}\sum_{j=1}^{\infty}(j+1)e^{-jP(\beta f)}\right)\right\} \\ 
&= \max\{ 0, \sum_{i\geq 2}a_{i} -2A\}.  
\end{align*} 
Furthermore,  for any fixed $j_0$, rewriting $$\left(1 + \sum_{j=1}^{\infty}(j+1)e^{\beta(a_{2}+...+a_{1+j})-jP(\beta f)}\right)$$
in the form $$ \left[1 + \sum_{j=1}^{j_0-1}(j+1)e^{\beta(a_{2}+...+a_{1+j})-jP(\beta f)}\right]  + \left[ \, \sum_{j=j_0}^{\infty}(j+1)e^{\beta(a_{2}+...+a_{1+j})-jP(\beta f)}\,\right]$$ we have
\begin{align*}
 \limsup_{\beta \to+\infty}&\frac{1}{\beta}\log\left(1 + \sum_{j=1}^{\infty}(j+1)e^{\beta(a_{2}+...+a_{1+j})-jP(\beta f)}\right) \\
 &= \max\left\{ 0\,,\,\, \limsup_{\beta \to+\infty}\frac{1}{\beta}\log\left( \sum_{j=j_0}^{\infty}(j+1)e^{\beta(a_{2}+...+a_{1+j})-jP(\beta f)}\right)\right\} \\ 
&\leq \max\left\{ 0\,,\,\, \limsup_{\beta \to+\infty}\frac{1}{\beta}\log\left( e^{\beta(a_{2}+...+a_{j_0})} \sum_{j=j_0}^{\infty}(j+1)e^{-jP(\beta f)}\right)\right\} \\ 
&= \max\{ 0, a_{2}+...+a_{j_0} -2A\}.  
\end{align*}
Thus, as we can consider $j_0$ large enough,  
\[\limsup_{\beta \to+\infty}\frac{1}{\beta}\log\left(1 + \sum_{j=1}^{\infty}(j+1)e^{\beta(a_{2}+...+a_{1+j})-jP(\beta f)}\right) \leq\max\{ 0, \sum_{j=2}^\infty a_{j} -2A\}.  \]
The conclusion is that
\[ \lim_{\beta \to+\infty}\frac{1}{\beta}\log\left(1 + \sum_{j=1}^{\infty}(j+1)e^{\beta(a_{2}+...+a_{1+j})-jP(\beta f)}\right) =\max\{ 0, \sum_{j=2}^\infty a_{j} -2A\}.\]
With similar arguments we obtain
\[\lim_{\beta \to+\infty}\frac{1}{\beta}\log\left(1+\sum_{j=1}^{\infty}e^{\beta(a_{2}+...+a_{1+j})-jP(\beta f)}	\right)=\max\{ 0, \sum_{j=2}^\infty a_{j} -A\}.\]
\end{proof}

Now we  show that the family of measures $\mu_\beta$ satisfies a L.D.P. and present the expression of the deviation function. We remark that, when the maximizing measure of a potential $f$ is unique, the deviation function in \cite{BLT} is equal to $R_+^\infty$. For the class of potentials that we consider in this section, we have $R_+^{\infty}(0^\infty)= R_+^{\infty}(1^\infty)=0$. However in the theorem below $I(0^\infty)\neq 0$, which means, $I\neq R_+^\infty$ (see also Theo. 3 and page 1343 in \cite{BLM}).

\begin{theorem} Let $f\in \mathbf{W}$ satisfying (\ref{eq14}).
Suppose $\sum_{j\geq 2}a_j < b+d+\sum_{j\geq 2}c_j$. Then $(\mu_\beta)_{\beta>0}$ satisfies a Large Deviation Principle with deviation function $I$ defined by $$I(0^\infty)= 	b+d+\sum_{j\geq 2}c_j -\sum_{j\geq 2}a_j, \,\,\,\,\,\,\,\,\,\, I(1^\infty)=0$$ and for any $x\in \{0,1\}^\mathbb{N}$,
\[I(x) = \left\{ \begin{array}{ll}
 R_+^n(x) + I(0^\infty) & \text{if}\,\, x=(x_1...x_n0^\infty) \\  
 R_+^n(x)  &  \text{if}\,\, x=(x_1...x_n1^\infty)\\
 +\infty & else
 \end{array} \right.,  \]
 where $R_+ = -f -U + U\circ\sigma$ and $U$ satisfies
\[U(0^{\infty})=0,\,\,\,\,\, U|_{[0^{q}1]} =  \max\left\{b+d+\sum_{j=2}^{\infty}c_{j}, \sum_{j=1}^{\infty}a_{q+j} \right\}, \, q\geq 1,
\]
\[U(1^{\infty}) =  b,\,\,\,\,\,\, and \,\,\,\,\,\,\,
U|_{[1^{q}0]} =   b  +  \sum_{j=1}^{\infty}c_{q+j},\, q\geq 1.
\]	
\end{theorem} 

\begin{proof}
First note that, with the hypothesis $\sum_{j\geq 2}a_j < b+d+\sum_{j\geq 2}c_j$, from Lemma \ref{pressure}, $A := \lim_{\beta \to +\infty}\frac{1}{\beta}\log(P(\beta f))= b+d+\sum_{j\geq 2}c_j$. Then, 
\begin{equation}\label{eq18}
\sum_{j\geq 2}a_j < A < \sum_{j\geq 2}c_j
\end{equation}
 and this function $U$ coincides with the one in Lemma \ref{Vlimit}, which means, $U=\lim_{\beta \to +\infty}\frac{1}{\beta}\log (H_\beta)$. Particularly, $R_+= -f -U + U\circ\sigma$ is the uniform limit of $-\frac{g_\beta}{\beta}$, when $\beta \to +\infty$.

\bigskip

\textit{claim:} 
\[\lim_{n\to+\infty}\lim_{\beta \to+\infty}\frac{1}{\beta}\log(\mu_\beta([0^n])) =   \sum_{j\geq 2} a_{j} -A  \]
and
\[\lim_{n\to+\infty}\lim_{\beta \to+\infty}\frac{1}{\beta}\log(\mu_\beta([1^n]))=0.\]
 
\bigskip
 
Indeed, from (\ref{eq18}), Lemma \ref{mu0n} and Lemma \ref{limits} we get, for n  large enough,
\begin{align*}
\lim_{\beta \to+\infty}&\frac{1}{\beta}\log(\mu_\beta([0^n])) = \lim_{\beta \to+\infty}\frac{1}{\beta}\log\left(\frac{S_0^n(\beta)}{S_0(\beta)+S_1(\beta)}  \right) \\
&= \lim_{\beta \to+\infty}\frac{1}{\beta}\log(S_0^n(\beta))-\max\left\{\lim_{\beta \to+\infty}\frac{1}{\beta}\log(S_0(\beta)), \lim_{\beta \to+\infty}\frac{1}{\beta}\log(S_1(\beta))  \right\} \\
&= \left[\sum_{j=2}^\infty a_{j} -2A\right] - \max\left\{ \max\{0,\sum_{j=2}^\infty a_{j} -2A\}\, ,\, -A\right\} \\
&=  \left[\sum_{j=2}^\infty a_{j} -2A\right]-(-A)=  \sum_{j=2}^\infty a_{j} - A
\end{align*}
and
\begin{align*}
\lim_{\beta \to+\infty}&\frac{1}{\beta}\log(\mu_\beta([1^n])) = \lim_{\beta \to+\infty}\frac{1}{\beta}\log\left(\frac{S_1^n(\beta)}{S_0(\beta)+S_1(\beta)}   \right)\\
&=\lim_{\beta \to+\infty}\frac{1}{\beta}\log(S_1^n(\beta))-\max\{\lim_{\beta \to+\infty}\frac{1}{\beta}\log(S_0(\beta)), \lim_{\beta \to+\infty}\frac{1}{\beta}\log(S_1(\beta))  \} \\
&= (-A)-(-A)   = 0.
\end{align*}
This concludes the proof of claim.

\bigskip

Let $I:X\to [0,+\infty]$ be defined by
\begin{align*}
I(0^\infty) &= - \lim_{n\to+\infty}\lim_{\beta \to+\infty}\frac{1}{\beta}\log(\mu_\beta([0^n]))= b+d+\sum_{j\geq 2}c_j -\sum_{j\geq 2}a_j,\\ 
I(1^\infty)&=- \lim_{n\to+\infty}\lim_{\beta \to+\infty}\frac{1}{\beta}\log(\mu_\beta([1^n]))=0
\end{align*}
 and, for any $x=(x_1x_2x_3...)$, $x\notin \{0^\infty, 1^\infty\}$,
\begin{equation} \label{eq17} 
I(x) = \inf_{y\in \{0^\infty,1^\infty\} } \liminf_{n\to+\infty}\left(R_{+}^{n}(x_{1}...x_{n}y) + I(y)\right). 
\end{equation} 
(It can be checked that equation (\ref{eq17}) is satisfied for $x=0^\infty$ and $x=1^\infty,$ but this is not necessary). 

\bigskip

 For any cylinder $k\subset \{0,1\}^\mathbb{N}$, we claim that
\begin{equation}\label{eq15} 
\lim_{\beta \to +\infty}\frac{1}{\beta}\log(\mu_{\beta}(k))= -\inf_{x\in k} I(x).
\end{equation} 
Indeed, for a given cylinder $k_0$, from $(\ref{eq6})$, the family $(\frac{1}{\beta}\log(\mu_\beta(k_0)))$ is bounded. If for a sequence $\beta_i \to +\infty$, we have that $\frac{1}{\beta_i}\log(\mu_{\beta_i}(k_0))$ converges, then (following the Remark 3., which appears below the Theorem \ref{teorema2}) for some subsequence $\beta_{i_j}$ of $\beta_i$ and for any cylinder $k\subset X$, we have 
\[\lim_{\beta_{i_j} \to+\infty}\frac{1}{\beta_{i_j}}\log(\mu_{\beta_{i_j}}(k))= -\inf_{x\in k} I(x).\]
Particularly, we get
\[\lim_{\beta_{i} \to +\infty}\frac{1}{\beta_{i}}\log(\mu_{\beta_{i}}(k_0))=\lim_{\beta_{i_j} \to+\infty}\frac{1}{\beta_{i_j}}\log(\mu_{\beta_{i_j}}(k_0))= -\inf_{x\in k_0} I(x). \]
This argument proves that 
\[\lim_{\beta \to +\infty}\frac{1}{\beta}\log(\mu_{\beta}(k_0))= -\inf_{x\in k_0} I(x), \]
which concludes the proof of (\ref{eq15}).
 
Now we study the function $I$. Given a point $x=(x_1x_2x_3...)$ for which $01$ occurs infinitely many times, we have $I(x)=+\infty$ because  $I\geq R_+^\infty$ and for each occurrence of $01$ in $x$,
\begin{align*}
R_+(01x_sx_{s+1}...)&= -b - U(01x_s...)+ U(1x_s...)   \\
&\geq -b-(b+d+\sum_{j\geq 2}c_j) + (b  +  \sum_{j\geq 2} c_{j})\\&=-b-d>0. 
\end{align*}
Then, from $(\ref{IandR})$,
\[I(x) = \left\{ \begin{array}{ll}
R_+^n(x) +  I(0^\infty) & \text{if}\,\, x=(x_1...x_n0^\infty) \\  
R_+^n(x) + I(1^\infty) &  \text{if}\,\, x=(x_1...x_n1^\infty)\\
+\infty & else
\end{array} \right. .  \]
\end{proof}

Assuming $\sum_{j\geq 2}c_j < b+d+\sum_{j\geq 2}a_j$ we get a symmetric result.

\begin{theorem} Let $f\in \mathbf{W}$ satisfying (\ref{eq14}).
	Suppose $\sum_{j\geq 2}a_j \geq b+d+\sum_{j\geq 2}c_j$ and $\sum_{j\geq 2}c_j \geq b+d+\sum_{j\geq 2}a_j$. Then, $\mu_\beta$ satisfies a Large Deviation Principle with deviation function  $I(x)=R_+^\infty(x)$. More precisely, $$I(0^\infty)= 	0, \,\,\,\,\,\,\,\,\,\, I(1^\infty)=0$$ and for any $x\in \{0,1\}^\mathbb{N}$,
	\[I(x) = \left\{ \begin{array}{ll}
	R_+^n(x)  & \text{if}\,\, x=(x_1...x_n0^\infty) \,\,or \,\, x=(x_1...x_n1^\infty)\\  
	+\infty & else
	\end{array} \right.,  \]
	where $R_+ = -f -U + U\circ\sigma$ and $U$ satisfies
	
	\begin{align*} 
	&U(0^{\infty})=0,\\
	&U(1^{\infty}) =  \frac{b}{2}-\frac{d}{2}   + \frac{1}{2}\sum_{j\geq 2} a_j - \frac{1}{2}\sum_{j\geq 2} c_j,\\
	&U(0^{q}1z) =   \sum_{j\geq 1}a_{q+j},\\
	&U(1^{q}0z) =  \frac{b}{2}-\frac{d}{2} + \frac{1}{2}\sum_{j\geq 2} a_j - \frac{1}{2}\sum_{j\geq 2} c_j  + \sum_{j\geq 1}c_{q+j} .
	\end{align*}
	
\end{theorem} 

\noindent
\textbf{Remark:} The above formulas for $U$ can have a more  symmetric expression if we add the constant $\frac{d}{2} + \frac{1}{2}\sum_{j\geq 2} c_j$. This is irrelevant when we consider the coboundary $U-U\circ\sigma$ in $R_+$. Thus we can consider $U$ defined by the formulas
\[U(0^{\infty})=\frac{d}{2} + \frac{1}{2}\sum_{j\geq 2} c_j,
\,\,\,\,\,\,\,U(1^{\infty}) =  \frac{b}{2}   + \frac{1}{2}\sum_{j\geq 2} a_j,\]
\[U(0^{q}1z) =   \frac{d}{2} + \frac{1}{2}\sum_{j\geq 2} c_j +\sum_{j\geq 1}a_{q+j},\,\,\,\,\,\,\,U(1^{q}0z) =  \frac{b}{2} + \frac{1}{2}\sum_{j\geq 2} a_j  + \sum_{j\geq 1}c_{q+j} .
\]

\begin{proof}
	We remark that in the present case (see Lemma \ref{pressure}), $$A=\lim_{\beta \to +\infty}\frac{1}{\beta}\log(P(\beta f)) =\frac{b+d}{2} +\sum_{j=1}^{\infty}\frac{a_{1+j}}{2} +\sum_{j=1}^{\infty}\frac{c_{1+j}}{2}. $$
		The proof of this theorem follows the same lines of the above one. We only  present some of the steps.
	
	\textit{first:} $I(0^\infty)=I(1^\infty)=0$. Indeed, as 
	\[\sum_{j\geq 2}a_j \geq A \geq 2A\,\,\,\,\, and \,\,\,\,\,\sum_{j\geq 2}c_j \geq A \geq 2A, \]
then,	from  lemmas \ref{mu0n} and \ref{limits} we get, for n  large enough, 
	\begin{align*} 
	\lim_{\beta \to+\infty}&\frac{1}{\beta}\log(\mu_\beta([0^n])) \\ 
	&= \lim_{\beta \to+\infty}\frac{1}{\beta}\log( S_0^n(\beta)) - \max\{\lim_{\beta \to+\infty}\frac{1}{\beta}\log( S_0(\beta))  , \lim_{\beta \to+\infty}\frac{1}{\beta}\log( S_1(\beta)) \} \\
	&=-A - \max\{-A, -A\} = 0
	\end{align*}
	and, similarly, 
	\[\lim_{\beta \to+\infty}\frac{1}{\beta}\log(\mu_\beta([1^n])) = -A - \max\{-A, -A\} = 0.\]

	\textit{second:} given a point $x=(x_1x_2x_3...)$ in which $01$ occurs infinitely many times, we have $I(x)=+\infty$. Indeed, as in this case we have $10$ occurring infinitely many times too, considering $R_+^\infty(x)$, for each occurrence of $01$ or $10$ in $x$, we get
	\begin{align*}
	R_+(01x_sx_{s+1}...)&\geq -b -\sum_{j\geq 2}a_j + ( \frac{b}{2}-\frac{d}{2} + \frac{1}{2}\sum_{j\geq 2} a_j - \frac{1}{2}\sum_{j\geq 2} c_j  + \sum_{j\geq 1}c_{1+j}) \\
	&=\frac{1}{2}(\sum_{j\geq 2} c_j - b-d-\sum_{j\geq 2}a_j)\geq 0; \\
	R_+(10x_sx_{s+1}...)&\geq -d - (\frac{b}{2}-\frac{d}{2} + \frac{1}{2}\sum_{j\geq 2} a_j + \frac{1}{2}\sum_{j\geq 2} c_j  ) + \sum_{j\geq 2} a_j\\
	&=\frac{1}{2}(\sum_{j\geq 2} a_j - b-d-\sum_{j\geq 2}c_j) \geq 0. 
	\end{align*}
	This numbers are not zero simultaneously, because their sum results in $-b-d$. Therefore $R_+^\infty(x)=+\infty$. 

\bigskip
	
	From this computations we conclude that the deviation function satisfies,
	\[I(x) = \left\{ \begin{array}{ll}
    0  & \text{if}\,\, x=0^\infty \,\, or \,\, x=1^\infty \\
	R_+^n(x)  & \text{if}\,\, x=(x_1...x_n0^\infty) \,\,or \,\, x=(x_1...x_n1^\infty)\\  
	+\infty & else
	\end{array} \right. \,.  \]
	
\end{proof}

\end{document}